\documentclass[11pt]{amsart}
\usepackage[margin=1in]{geometry}

\usepackage{amsmath}
\usepackage{amscd}
\usepackage{tikz-cd}
\usepackage{amssymb}
\usepackage[mathscr]{euscript}
\usepackage{enumerate}
\usepackage{hyperref}

\DeclareMathOperator{\End}{End}

\DeclareMathOperator{\Spec}{Spec}
\DeclareMathOperator{\Ann}{Ann}
\DeclareMathOperator{\Ass}{Ass}
\DeclareMathOperator{\hgt}{ht}
\DeclareMathOperator{\gr}{gr}
\DeclareMathOperator{\im}{im}
\DeclareMathOperator{\injdim}{injdim}

\DeclareMathOperator{\Supp}{Supp}

\usepackage{amsthm}
\usepackage[all,cmtip]{xy}

\theoremstyle{plain}
\newtheorem{thm}{Theorem}[section]

\newtheorem*{mainthm}{Main Theorem}
\newtheorem{prop}[thm]{Proposition}
\newtheorem{lem}[thm]{Lemma}
\newtheorem{cor}[thm]{Corollary}

\theoremstyle{definition}
\newtheorem{definition}[thm]{Definition}

\theoremstyle{remark}
\newtheorem{remark}[thm]{Remark}

\newtheorem{question}[thm]{Question}

\numberwithin{equation}{thm}

\newcommand{\pref}[1]{(\ref{#1})}

\newcommand{\fp}{\mathfrak{p}}
\newcommand{\fq}{\mathfrak{q}}
\newcommand{\fs}{\mathfrak{s}}
\newcommand{\fm}{\mathfrak{m}}

\begin{document}

\title[Injective dimension of $F$- and $D$-modules]{A dichotomy for the injective dimension of $F$-finite $F$-modules and holonomic $D$-modules}

\author{Nicholas Switala and Wenliang Zhang}

\address{Department of Mathematics, Statistics, and Computer Science\\University of Illinois at Chicago\\322 SEO (M/C 249)\\851 S. Morgan Street\\Chicago, IL 60607}

\email{nswitala@uic.edu, wlzhang@uic.edu}

\dedicatory{Dedicated to Gennady Lyubeznik on the occasion of his sixtieth birthday}



\thanks{The first author acknowledges NSF support through grant DMS-1604503. The second author is partially supported by the NSF through grant DMS-1606414.}

\begin{abstract}
Let $M$ be either a holonomic $D$-module over a formal power series ring with coefficients in a field of characteristic zero, or an $F$-finite $F$-module over a noetherian regular ring of characteristic $p > 0$. We prove that $\injdim_R M$ enjoys a dichotomy property: it has only two possible values, $\dim \Supp_R M - 1$ or $\dim \Supp_R M$.
\end{abstract}

\maketitle

\section{Introduction}\label{Introduction}

This paper is concerned with the injective dimension of two kinds of modules: $D$-modules $M$ over a formal power series ring $R = k[[x_1, \ldots, x_n]]$ where $k$ is a field of characteristic zero, and $F$-modules $M$ over a noetherian regular ring $R$ of characteristic $p>0$. In both cases, the upper bound
\begin{equation}\label{basicbound}
\injdim_R M \leq \dim \Supp_R M
\end{equation}
holds by the foundational work of Lyubeznik (\cite[Theorem 2.4(b)]{LyubeznikFinitenessLocalCohomology} in the $D$-module case and \cite[Theorem 1.4]{LyubeznikFModulesApplicationsToLocalCohomology} in the $F$-module case). Here $\injdim_R M$ denotes the injective dimension of $M$ as an $R$-module. Lyubeznik's proof also shows that \pref{basicbound} is true if $R$ is a polynomial ring instead of a formal power series ring. In the special case of local cohomology in positive characteristic, \pref{basicbound} is due originally to Huneke and Sharp \cite[Corollary 3.9]{HunekeSharpBassnumbersoflocalcohomologymodules}; in equicharacteristic zero, Lyubeznik shows further \cite[Theorem 3.4(b)]{LyubeznikFinitenessLocalCohomology} that \pref{basicbound} holds for the local cohomology of any noetherian regular ring. 

In either setting, if $\fp \subseteq R$ is a prime ideal of dimension $d$ and $E(R/\fp)$ is the $R$-module injective hull of $R/\fp$, then $E(R/\fp)$ is a $D$-module (resp. $F$-module) with injective dimension zero whose support has dimension $d$. It is clear from this example that without imposing further hypotheses, there does not exist a nontrivial \emph{lower} bound for $\injdim_R M$ in terms of $\dim \Supp_R M$. 

Our main result is the following theorem.

\begin{mainthm}[Theorems \ref{holonomic over power series} and \ref{F-finite over regular char p}]\label{maintheorem}
Let $M$ be either a holonomic $D$-module over $R = k[[x_1, \ldots, x_n]]$ where $k$ is a field of characteristic zero, or an $F$-finite $F$-module over a noetherian regular ring $R$ of characteristic $p > 0$. Then $\injdim_R M \geq \dim \Supp_R M - 1$.
\end{mainthm}

Our Main Theorem combined with the upper bound \pref{basicbound} shows that $\injdim_R M$ (with $M$ as in the theorem) enjoys a dichotomy property: it has only two possible values, either $\dim \Supp_R M - 1$ or $\dim \Supp_R M$.

In the case of a polynomial ring $R = k[x_1, \ldots, x_n]$ over a characteristic-zero field, Puthenpurakal \cite[Corollary 1.2]{PuthenpurakaInjectiveResolutionofLC} has shown that $\injdim_R M = \dim \Supp_R M$ whenever $M$ is a local cohomology module of $R$. This result was strengthened by Zhang \cite[Theorem 4.5]{zhanginjdim} who established this equality for all holonomic $D$-modules over polynomial rings, as well as for all $F$-finite $F$-modules over certain regular rings (finitely generated algebras over an infinite, positive-characteristic field). In the case of a formal power series ring, or in the general case of a positive-characteristic regular local ring, this equality need not hold: indeed, in both cases, the injective hull of $R/\fp$ where $\fp$ is a prime ideal of dimension one provides a counterexample (see Remark \ref{bound is sharp} below). 

The proof of our Main Theorem appears in sections \ref{injdim over power series} and  \ref{injdim of F-modules}, following the most technical part of the paper: section \ref{minimal resolutions}, an in-depth study of the last terms of minimal injective resolutions over the rings considered in our Main Theorem as well as their localizations. One key observation is that the assumption that $R$ is Jacobson in \cite[Theorem 3.3]{zhanginjdim} can be weakened; to this end, we introduce a notion of {\it pseudo-Jacobson} rings in section \ref{Pseudo-Jacobson rings}.

During the preparation of this paper, we were made aware of the article \texttt{arXiv:1603.06639v1}, which investigates the injective dimension of local cohomology modules $H^j_I(R)$ when $R$ is a formal power series ring in characteristic zero and the dimension of the support of $H^j_I(R)$ is at most 4. In November 2017 an updated version, \texttt{arXiv:1603.06639v2}, appeared; it investigates the injective dimension of $F$-finite $F$-modules over a regular local ring in characteristic $p > 0$ and modules of the form $(H^j_I(R))_g$ over a regular local ring $R$ in characteristic zero (here $g \in R$). The approach in our paper is different: in order to investigate the injective dimension of holonomic $D$-modules, we introduce and study the notion of pseudo-Jacobson rings. Such an approach works well for both holonomic $D$-modules and $F$-finite $F$-modules, further illustrating the nice parallel between these two classes of modules.

\subsection*{Acknowledgments}
The authors thank Mel Hochster and Gennady Lyubeznik for helpful discussions.

\section{$D$-module and $F$-module preliminaries}\label{FandDmodules}

In this section, we review some basic notions concerning $D$-modules, $F$-modules, and local cohomology that will be needed throughout the paper. We begin with general notation and conventions.

Throughout the paper, a \emph{ring} is commutative with $1$ unless otherwise specified, and a \emph{local ring} is always noetherian.

If $R$ is a noetherian ring and $M$ is an $R$-module, we will denote the minimal injective resolution of $M$ as an $R$-module by $E^{\bullet}_R(M)$ (or $E^{\bullet}(M)$ if $R$ is understood). The $R$-module $E^0_R(M) = E^0(M)$ (the \emph{injective hull} of $M$) will simply be denoted $E_R(M)$ (or $E(M)$). We denote the set of associated primes of $M$ by $\Ass M$ or $\Ass_R M$. 

If $R$ is a ring and $S \subseteq R$ is a multiplicative subset, we will use without further comment the one-to-one correspondence between prime ideals of $S^{-1}R$ and prime ideals of $R$ that do not meet $S$. In particular, if we write ``let $S^{-1}\fp$ be a prime ideal of $S^{-1}R$'', it is to be understood that $\fp \subseteq R$ is a prime ideal and $\fp \cap S = \emptyset$.

\subsection{$D$-modules} Our basic references for $D$-modules are EGA \cite{EGAIV4} and the book \cite{BjorkBookRIngDiffOperators} of Bj\"{o}rk.

Let $R$ be a ring and $k \subseteq R$ a subring. We denote by $D(R,k)$ (or simply $D$, if $R$ and $k$ are understood) the (usually non-commutative) ring of $k$-linear differential operators on $R$, which is a subring of $\End_k(R)$. This ring is recursively defined as follows \cite[\S 16]{EGAIV4}.  A differential operator $R \rightarrow R$ of order zero is multiplication by an element of $R$.  Supposing that differential operators of order $\leq j-1$ have been defined, $d \in \End_k(R)$ is said to be a differential operator of order $\leq j$ if, for all $r \in R$, the commutator $[d,r] \in \End_k(R)$ is a differential operator of order $\leq j-1$, where $[d,r] = dr - rd$ (the products being taken in $\End_k(R)$).  We write $D^j(R)$ for the set of differential operators on $R$ of order $\leq j$ and set $D(R,k) = \cup_j D^j(R)$.  If $d \in D^j(R)$ and $d' \in D^{\l}(R)$, it is easy to prove by induction on $j + \l$ that $d' \circ d \in D^{j+\l}(R)$, so $D(R,k)$ is a ring.

The most important case for us will be where $k$ is a field of characteristic zero and $R = k[[x_1, \ldots, x_n]]$ is a formal power series ring over $k$. In this case, the ring $D$, viewed as a left $R$-module, is freely generated by monomials in the partial differentiation operators $\partial_1 = \frac{\partial}{\partial x_1}, \ldots, \partial_n = \frac{\partial}{\partial x_n}$ (\cite[Theorem 16.11.2]{EGAIV4}: here the characteristic-zero assumption is necessary).  This ring has an increasing filtration $\{D(\nu)\}$, called the \emph{order filtration}, where $D(\nu)$ consists of those differential operators of order $\leq \nu$.  The associated graded object $\gr(D) = \oplus D(\nu)/D(\nu-1)$ with respect to this filtration is isomorphic to $R[\zeta_1, \ldots, \zeta_n]$ (a commutative ring), where $\zeta_i$ is the image of $\partial_i$ in $D(1)/D(0) \subseteq \gr(D)$.  

By a \emph{$D$-module} we always mean a \emph{left} module over the ring $D$ unless otherwise specified. If $M$ is a finitely generated $D$-module, there exists a \emph{good filtration} $\{M(\nu)\}$ on $M$, meaning that $M$ becomes a filtered left $D$-module with respect to the order filtration on $D$ \emph{and} $\gr(M) = \oplus M(\nu)/M(\nu-1)$ is a finitely generated $\gr(D)$-module.  We let $J$ be the radical of $\Ann_{\gr(D)} \gr(M) \subseteq \gr(D)$ and set $d(M) = \dim \gr(D)/J$. The ideal $J$, and hence the number $d(M)$, is independent of the choice of good filtration on $M$.  By \emph{Bernstein's theorem}, if $M \neq 0$ is a finitely generated $D$-module, we have $n \leq d(M) \leq 2n$.  

\begin{definition}\label{holonomic}
Let $M$ be a finitely generated $D$-module. We say that $M$ is \emph{holonomic} if $M = 0$ or $d(M) = n$.
\end{definition} 

The ring $R$ itself is a holonomic $D$-module. More generally, local cohomology modules of $R$ are holonomic $D$-modules (Proposition \ref{local coh omnibus}(c)). The ring $D$ is a holonomic $D$-module if and only if $n=0$ (so $D = k$), since $d(D) = 2n$. 

We collect in the following proposition the basic results on $D$-modules that we will use below.

\begin{prop}\label{D-modules omnibus}
Let $R = k[[x_1, \ldots, x_n]]$ where $k$ is a field of characteristic zero, and let $M$ be a $D(R,k)$-module.
\begin{enumerate}[(a)]
\item $\injdim_R M \leq \dim \Supp_R M$ \cite[Theorem 2.4(b)]{LyubeznikFinitenessLocalCohomology};
\item If $S \subseteq R$ is a multiplicative subset, $S^{-1}M$ is both a $D$-module and a $D(S^{-1}R, k)$-module; and if $M$ is of finite length as a $D$-module, $S^{-1}M$ is of finite length as a $D(S^{-1}R, k)$-module (\cite[Proposition 2.5]{zhanginjdim}; this is true for any domain $R$ and subring $k$);
\item If $M$ is finitely generated as a $D$-module (in particular, if $M$ is holonomic), then $M$ has finitely many associated primes as an $R$-module \cite[Theorem 2.4(c)]{LyubeznikFinitenessLocalCohomology};
\item If $M$ is holonomic, then $M$ is of finite length as a $D$-module \cite[Theorem 2.7.13]{BjorkBookRIngDiffOperators};
\item If $0 \rightarrow M' \rightarrow M \rightarrow M'' \rightarrow 0$ is a short exact sequence of $D$-modules and $D$-linear maps, then $M$ is holonomic (resp. finite length) if and only if $M'$ and $M''$ are holonomic (resp. finite length).
\end{enumerate}
\end{prop}

As remarked in \cite[2.2(c)]{LyubeznikFinitenessLocalCohomology}, a proof of the ``holonomic'' part of Proposition \ref{D-modules omnibus}(e) is analogous to the proof of \cite[Proposition 1.5.2]{BjorkBookRIngDiffOperators} (the ``finite length'' part is a well-known fact about modules over any ring). We note that part (b) of the proposition does \emph{not} assert that if $M$ is of finite length as a $D$-module, so is $S^{-1}M$. This is not known even in the case where $S^{-1}R = R_f$ for a single element $f \in R$. Part (b) only makes the weaker claim that $S^{-1}M$ is of finite length as a $D(S^{-1}R, k)$-module.

\subsection{$F$-modules} Our basic reference for $F$-modules is the paper \cite{LyubeznikFModulesApplicationsToLocalCohomology} of Lyubeznik in which they were introduced.

Let $R$ be a noetherian regular ring of characteristic $p>0$. Let $F_R$ denote the Peskine-Szpiro functor: 
\[F_R(M)=R' \otimes_R M\]
for each $R$-module $M$, where $R'$ denotes the $R$-module that is the same as $R$ as a left $R$-module and whose right $R$-module structure is given by $r'\cdot r=r^pr'$ for all $r'\in R'$ and $r\in R$.

\begin{definition}[Definitions 1.1, 1.9 and 2.1 in \cite{LyubeznikFModulesApplicationsToLocalCohomology}]
\label{definition: F-modules}
An $F_R$-{\it module} (or \emph{$F$-module}, if $R$ is understood) is an $R$-module $M$ equipped with an $R$-linear isomorphism $\theta_{M}:M\to F_R(M)$. 

A \emph{homomorphism} between $F_R$-modules $(M,\theta_{M})$ and $(N,\theta_N)$ is an $R$-linear map $\varphi:M\to N$ such that the following diagram commutes:
\[\xymatrix{
M \ar[r]^{\theta_{M}} \ar[d]^{\varphi} & F_R(M) \ar[d]^{F_R(\varphi)}\\
N \ar[r]^{\theta_{N}} & F_R(N).
}\]

A {\it generating morphism} of an $F_R$-module $(M,\theta_{M})$ is an $R$-linear map $\beta:M'\to F_R(M')$, where $M'$ is an $R$-module, such that the direct limit of the diagram 
\[
\xymatrix{
M' \ar[r] \ar[d] & F_R(M') \ar[r]^{F_R(\beta)} \ar[d]^{F_R(\beta)} & F^2_R(M')\ar[r] \ar[d]^{F^2_R(\beta)} &\cdots \\
F_R(M') \ar[r]^{F_R(\beta)} & F^2_R(M') \ar[r]^{F^2_R(\beta)} & F^3_R(M')\ar[r] &\cdots 
}\] 
is the map $\theta_{M}:M\to F_R(M)$.

An $F_R$-module $M$ is called $F_R$-{\it finite} (or \emph{$F$-finite}) if it admits a generating morphism $\beta:M'\to F_R(M')$ such that $M'$ is a finitely generated $R$-module. 
\end{definition}

The counterpart to Proposition \ref{D-modules omnibus} for $F$-modules is the following.

\begin{prop}\label{F-modules omnibus}
Let $R$ be a noetherian regular ring of characteristic $p > 0$, and let $M$ be an $F$-module.
\begin{enumerate}[(a)]
\item $\injdim_R M \leq \dim \Supp_R M$ \cite[Theorem 1.4]{LyubeznikFModulesApplicationsToLocalCohomology};
\item The minimal injective resolution $E^{\bullet}(M)$ is a complex of $F$-modules and $F$-module morphisms \cite[Example 1.2(b'')]{LyubeznikFModulesApplicationsToLocalCohomology};
\item If $S \subseteq R$ is a multiplicative subset, then $S^{-1}M$ is an $F_{S^{-1}R}$-module; and if $M$ is $F$-finite, then $S^{-1}M$ is $F_{S^{-1}R}$-finite (both statements follow from \cite[Remark 1.0(i)]{LyubeznikFModulesApplicationsToLocalCohomology});
\item If $M$ is $F$-finite, $M$ has finitely many associated primes as an $R$-module \cite[Theorem 2.12(a)]{LyubeznikFModulesApplicationsToLocalCohomology};
\item If $0 \rightarrow M' \rightarrow M \rightarrow M'' \rightarrow 0$ is a short exact sequence of $F$-modules and $F$-module morphisms, then $M$ is $F$-finite if and only if $M'$ and $M''$ are $F$-finite \cite[Theorem 2.8]{LyubeznikFModulesApplicationsToLocalCohomology};
\item If $E$ is an injective $R$-module, $E$ is an $F$-module \cite[Proposition 1.5]{HunekeSharpBassnumbersoflocalcohomologymodules}.
\end{enumerate}
\end{prop}

\subsection{Local cohomology}
We will also make use of local cohomology modules, for whose definition and basic properties we refer to \cite{LCBookBrodmannSharp}. The most important facts about local cohomology we will use are the following.

\begin{prop}\label{local coh omnibus}
\begin{enumerate}[(a)]
\item Local cohomology commutes with flat base change: if $R \rightarrow S$ is a flat homomorphism of noetherian rings, $I \subseteq R$ is an ideal, and $M$ is an $R$-module, then $H^i_{IS}(S \otimes_R M) \cong S \otimes H^i_I(M)$ as $S$-modules for all $i \geq 0$ \cite[Theorem 4.3.2]{LCBookBrodmannSharp}. In particular, local cohomology commutes with localization and completion.
\item If $R = k[[x_1, \ldots, x_n]]$ where $k$ is a field of characteristic zero, $M$ is a holonomic $D$-module, and $I \subseteq R$ is an ideal, then for all $i \geq 0$, the local cohomology module $H^i_I(M)$ is a holonomic $D$-module; in particular, $H^i_I(R)$ is a holonomic $D$-module \cite[2.2(d)]{LyubeznikFinitenessLocalCohomology}.
\item If $R$ is a noetherian regular ring of characteristic $p > 0$, $M$ is an $F$-finite $F$-module, and $I \subseteq R$ is an ideal, then for all $i \geq 0$, the local cohomology module $H^i_I(M)$ is an $F$-finite $F$-module; in particular, $H^i_I(R)$ is an $F$-finite $F$-module \cite[Proposition 2.10]{LyubeznikFModulesApplicationsToLocalCohomology}.
\item If $R$ is a Gorenstein ring and $\fm \subseteq R$ is a maximal ideal of height $n$, then $H^n_{\fm}(R) \cong E(R/\fm)$ as $R$-modules \cite[Lemma 11.2.3]{LCBookBrodmannSharp}.
\end{enumerate}
\end{prop}

Part (d) of Proposition \ref{local coh omnibus} is stated in \cite{LCBookBrodmannSharp} only in the case where $R$ is a Gorenstein \emph{local} ring, but the same result is true for maximal ideals in arbitrary Gorenstein rings, with the same proof: if $E^{\bullet}$ is the minimal injective resolution of $R$ as a module over itself, then $\Gamma_{\fm}(E^{\bullet})$, which computes the local cohomology of $R$ supported at $\fm$, is simply $E(R/\fm)$ concentrated in degree $n$.

Recall that if $M$ is a module over a noetherian ring $R$, $E^{\bullet}(M)$ is its minimal injective resolution, and $\fp \subseteq R$ is a prime ideal, then the \emph{Bass number} $\mu^i(\fp, M)$ is the (possibly infinite) number of copies of the indecomposable injective hull $E(R/\fp)$ occurring as direct summands of $E^i(M)$ (see \cite[\S 18]{MatsumuraCRT} for properties of Bass numbers, including their well-definedness). In particular, to say that $\mu^i(\fp, M) > 0$ is to say that $E(R/\fp)$ is a summand of $E^i(M)$, which implies that $\fp \in \Supp_R E^i(M)$. 

If $\fp \subseteq R$ is a prime ideal, the $R$-module $E(R/\fp)$ is naturally an $R_{\fp}$-module isomorphic to $E_{R_{\fp}}(R_{\fp}/\fp R_{\fp})$. We will use this fact repeatedly. For now, we remark that in conjunction with Proposition \ref{local coh omnibus}(d), this fact implies that $E(R/\fp)$ is a $D(R,k)$-module whenever $R$ is a Gorenstein ring and $k \subseteq R$ is a subring; indeed, since $R_{\fp}$ is a Gorenstein local ring, we have
\[
E(R/\fp) \cong E_{R_{\fp}}(R_{\fp}/\fp R_{\fp}) \cong H^{\hgt \fp}_{\fp R_{\fp}}(R_{\fp}) \cong (H^{\hgt \fp}_{\fp}(R))_{\fp}
\]
as $R_{\fp}$-modules, so since $H^{\hgt \fp}_{\fp}(R)$ is a $D$-module by Proposition \ref{local coh omnibus}(b), its localization $E(R/\fp)$ is as well. (In the $F$-module case, if $R$ is a regular local ring of characteristic $p>0$ and $\fp \subseteq R$ is a prime ideal, then $E(R/\fp)$ is an $F$-module by Proposition \ref{F-modules omnibus}(f).) 

Finally, we will need to make use of a lemma of Lyubeznik on Bass numbers and local cohomology.

\begin{lem}\cite[Lemma 1.4]{LyubeznikFinitenessLocalCohomology}\label{Bass numbers of local coh}
Let $R$ be a noetherian ring, let $\fp \subseteq R$ be a prime ideal, and let $M$ be an $R$-module such that the $R_{\fp}$-module $(H^i_{\fp}(M))_{\fp}$ is injective for all $i \geq 0$. 
\begin{enumerate}[(a)]
\item All differentials in the complex $(\Gamma_{\fp}(E^{\bullet}(M)))_{\fp}$ of $R_{\fp}$-modules are zero.
\item For all $i \geq 0$, the Bass numbers $\mu^i(\fp, M)$ and $\mu^0(\fp, H^i_{\fp}(M))$ are equal.
\end{enumerate}
\end{lem}

\section{Localizations of $D$-modules}\label{inj dim of localization}

Throughout this section, let $R = k[[x_1, \ldots, x_n]]$ where $k$ is a field of characteristic zero, and let $D = D(R,k)$. The goal of this section is to prove the following generalization of Proposition \ref{D-modules omnibus}(a).

\begin{thm}\label{inj dim bound for arbitrary localization}
Let $M$ be a $D$-module, and let $S \subseteq R$ be a multiplicative subset. Then
\[
\injdim_{S^{-1}R} S^{-1}M \leq \dim \Supp_{S^{-1}R} S^{-1}M.
\]
\end{thm}

In fact, it suffices to prove the following weaker statement.

\begin{prop}\label{inj dim bound for localization}
Let $M$ be a $D$-module, and let $\fp \in \Supp_R M$. Then 
\[
\injdim_{R_{\fp}} M_{\fp} \leq \dim \Supp_{R_{\fp}} M_{\fp}.
\]
\end{prop}

\begin{proof}[Proof that Proposition \ref{inj dim bound for localization} implies Theorem \ref{inj dim bound for arbitrary localization}]
Let $M$ and $S \subseteq R$ be given, and let $t$ be the injective dimension $\injdim_{S^{-1}R} S^{-1}M$. Since $E^t_{S^{-1}R}(S^{-1}M) \neq 0$, there exists a prime ideal $S^{-1} \fp \subseteq S^{-1}R$ such that $\mu^t(S^{-1}\fp, S^{-1}M) > 0$. Since $S^{-1} \fp$ belongs to the support of $E^t_{S^{-1}R}(S^{-1}M)$, if we localize the complex $E^{\bullet}_{S^{-1}R}(S^{-1}M)$ at $S^{-1}\fp$, its length remains the same. But this new complex is the minimal injective resolution of $(S^{-1}M)_{S^{-1} \fp} = M_{\fp}$ as an $R_{\fp}$-module, so we have
\[
t = \injdim_{S^{-1}R} S^{-1}M = \injdim_{R_{\fp}} M_{\fp} \leq \Supp_{R_{\fp}} M_{\fp} \leq \dim \Supp_{S^{-1}R} S^{-1}M,
\]
where the first inequality holds since we have assumed Proposition \ref{inj dim bound for localization}. This completes the proof.
\end{proof}

The proof of Proposition \ref{inj dim bound for localization} below proceeds similarly to that of \cite[Theorem 3.4(b)]{LyubeznikFinitenessLocalCohomology}, an analogous statement for local cohomology modules over more general rings. We first need a couple of lemmas.

\begin{lem}\label{completion is a D-module}
Let $\fp$ be a prime ideal of $R$. Then the $\fp R_{\fp}$-adic completion $\widehat{R_{\fp}}$ of $R_{\fp}$ is isomorphic to a formal power series ring $K[[z_1, \ldots, z_c]]$ where $K$ is a field of characteristic zero and $c = \hgt \fp$, and the $\widehat{R_{\fp}}$-module $\widehat{R_{\fp}} \otimes_{R_{\fp}} M_{\fp}$ is in fact a $D(\widehat{R_{\fp}}, K)$-module.
\end{lem}

\begin{proof}
The statement about the form of the ring $\widehat{R_{\fp}}$ is simply Cohen's structure theorem, since $R_{\fp}$ is a regular local ring. The second statement is essentially included in the proof of \cite[Corollary 8]{LyubeznikCharFreeFiniteness} (see also the proof of \cite[Theorem 2.4]{LyubeznikFinitenessLocalCohomology}), so we omit most details, contenting ourselves with the following outline. There exist derivations $\delta_i: R_{\fp} \rightarrow R_{\fp}$ for $1 \leq i \leq c$ such that, upon passing to the completion, each $\delta_i$ induces the $K$-linear derivation $\partial_i = \frac{\partial}{\partial z_i}$ on $\widehat{R_{\fp}}$. We define the $D(\widehat{R_{\fp}}, K)$-module structure on $\widehat{R_{\fp}} \otimes_{R_{\fp}} M_{\fp}$ as follows: if $\hat{r}, \hat{s} \in \widehat{R_{\fp}}$ and $\mu \in M_{\fp}$, then $\hat{r} \cdot (\hat{s} \otimes \mu) = \hat{r}\hat{s} \otimes \mu$ and $\partial_i \cdot (\hat{s} \otimes \mu) = \partial_i(\hat{s}) \otimes \mu + \hat{s} \otimes \delta_i(\mu)$. It is easy to see that, for $1 \leq i \leq c$ and all $\hat{r} \in R_{\fp}$, the actions of $\partial_i \hat{r} - \hat{r} \partial_i$ and $\partial_i(\hat{r})$ on $\hat{s} \otimes \mu$ are the same.
\end{proof}

\begin{lem}\label{localized local coh is injective}
For all prime ideals $\fp$ of $R$ and all $i \geq 0$, the $R_{\fp}$-module $(H^i_{\fp}(M))_{\fp}$ is injective.
\end{lem}

\begin{proof}
First recall that $(H^i_{\fp}(M))_{\fp} \cong H^i_{\fp R_{\fp}}(M_{\fp})$ as $R_{\fp}$-modules. Since $H^i_{\fp R_{\fp}}(M_{\fp})$ is supported only at $\fp R_{\fp}$, every element of this module is annihilated by some power of $\fp R_{\fp}$, and therefore $H^i_{\fp R_{\fp}}(M_{\fp})$ is an $\widehat{R_{\fp}}$-module, where $\widehat{R_{\fp}}$ is the $\fp R_{\fp}$-adic completion of $R_{\fp}$: this $\widehat{R_{\fp}}$-module may be identified with $\widehat{R_{\fp}} \otimes_{R_{\fp}} H^i_{\fp R_{\fp}}(M_{\fp})$. The extension $R_{\fp} \rightarrow \widehat{R_{\fp}}$ is flat, so since local cohomology commutes with flat base change, we have
\[
\widehat{R_{\fp}} \otimes_{R_{\fp}} H^i_{\fp R_{\fp}}(M_{\fp}) \cong H^i_{\fp \widehat{R_{\fp}}}(\widehat{R_{\fp}} \otimes_{R_{\fp}} M_{\fp})
\]
as $\widehat{R_{\fp}}$-modules. As in Lemma \ref{completion is a D-module}, $\widehat{R_{\fp}} \cong K[[z_1, \ldots, z_c]]$ where $K$ is a field of characteristic zero. By that lemma, $\widehat{R_{\fp}} \otimes_{R_{\fp}} M_{\fp}$, and hence $H^i_{\fp \widehat{R_{\fp}}}(\widehat{R_{\fp}} \otimes_{R_{\fp}} M_{\fp})$, is a $D(\widehat{R_{\fp}}, K)$-module. Since $\widehat{R_{\fp}}$ is a formal power series ring over $K$, we have
\[
\injdim_{\widehat{R_{\fp}}} H^i_{\fp \widehat{R_{\fp}}}(\widehat{R_{\fp}} \otimes_{R_{\fp}} M_{\fp}) \leq \dim \Supp_{\widehat{R_{\fp}}} H^i_{\fp \widehat{R_{\fp}}}(\widehat{R_{\fp}} \otimes_{R_{\fp}} M_{\fp}) = 0,
\]
where the inequality is Proposition \ref{D-modules omnibus}(a) and the equality holds because $H^i_{\fp \widehat{R_{\fp}}}(\widehat{R_{\fp}} \otimes_{R_{\fp}} M_{\fp})$ is supported only at the maximal ideal $\fp \widehat{R_{\fp}}$. Therefore $H^i_{\fp \widehat{R_{\fp}}}(\widehat{R_{\fp}} \otimes_{R_{\fp}} M_{\fp})$, which we have identified with $H^i_{\fp R_{\fp}}(M_{\fp})$, is injective as an $\widehat{R_{\fp}}$-module; since $R_{\fp} \rightarrow \widehat{R_{\fp}}$ is flat, $H^i_{\fp R_{\fp}}(M_{\fp})$ is injective over $R_{\fp}$ as well, completing the proof.
\end{proof}

\begin{proof}[Proof of Proposition \ref{inj dim bound for localization}]
We proceed by induction on $\dim \Supp_{R_{\fp}} M_{\fp}$. If $\fp$ is a minimal prime of $M$, then this dimension is zero and we must show that $M_{\fp}$ is injective as an $R_{\fp}$-module. Since $\fp$ is minimal in $\Supp_R M$, every element of $M_{\fp}$ is annihilated by some power of $\fp R_{\fp}$, and so $M_{\fp}$ is a module over the $\fp R_{\fp}$-adic completion $\widehat{R_{\fp}}$ of $R_{\fp}$: this module may be identified with $\widehat{R_{\fp}} \otimes_{R_{\fp}} M_{\fp}$. By the same reasoning used in the proof of Lemma \ref{localized local coh is injective}, $M_{\fp}$ is injective over $\widehat{R_{\fp}}$ and therefore over $R_{\fp}$. Now suppose that $\dim \Supp_{R_{\fp}} M_{\fp} > 0$. Let $E^{\bullet}(M_{\fp})$ denote the minimal injective resolution of $M_{\fp}$ as an $R_{\fp}$-module. By the inductive hypothesis, if $\fq \subset \fp$ and $\fq \in \Supp_R M$, we have
\[
\injdim_{R_{\fq}} M_{\fq} \leq \dim \Supp_{R_{\fq}} M_{\fq} < \dim \Supp_{R_{\fp}} M_{\fp},
\]
so if $i \geq \dim \Supp_{R_{\fp}} M_{\fp}$, $E^i(M_{\fp})$ is supported only at $\fp R_{\fp}$. By Lemma \ref{localized local coh is injective}, $(H^i_{\fp}(M))_{\fp}$ is an injective $R_{\fp}$-module for all $i \geq 0$. By Lemma \ref{Bass numbers of local coh}(a), the differentials
\[
E^i(M_{\fp}) = \Gamma_{\fp R_{\fp}}(E^i(M_{\fp}))\rightarrow \Gamma_{\fp R_{\fp}}(E^{i+1}(M_{\fp})) = E^{i+1}(M_{\fp})
\]
are zero for all $i \geq \dim \Supp_{R_{\fp}} M_{\fp}$. By the minimality of the resolution, $E^i(M_{\fp})$ itself is zero for all such $i$, completing the proof.
\end{proof}

\section{Pseudo-Jacobson rings}\label{Pseudo-Jacobson rings}

Recall that a ring $R$ is said to be \emph{Jacobson} if every prime ideal of $R$ is equal to the intersection of the maximal ideals containing it, and that if $R$ is a Jacobson ring, so also is every quotient $R/I$ of $R$. It is not hard to see from this that every non-maximal prime ideal of $R$ must be contained in \emph{infinitely many} distinct maximal ideals. It is this weaker statement that will be important for us; hence we make the following definition.

\begin{definition}\label{pseudo-Jacobson ring}
A commutative ring $R$ is called \emph{pseudo-Jacobson} if every non-maximal prime ideal $\fp$ of $R$ is contained in infinitely many distinct maximal ideals.
\end{definition}

Pseudo-Jacobson rings will arise for us in the following way: if $R$ is a regular local ring and $f$ is a non-unit of $R$, then unless $R$ is of very small dimension, the localization $R_f$ is pseudo-Jacobson. This follows from Proposition \ref{localization is pseudo-Jacobson}(a) below; the next preliminary results are given with this result in mind. 

\begin{lem}\label{intersection of height 1 primes}
Let $(R, \fm)$ be a local domain of dimension $d \geq 2$. Then
\[
\bigcap_{\substack{\fp \subseteq R \, \mathrm{prime}\\{\hgt \fp = 1}}} \fp = 0.
\]
\end{lem}

\begin{proof}
Suppose otherwise, and let $f \neq 0$ belong to the displayed intersection. Then every height $1$ prime ideal $\fp$ is a minimal prime of $f$. Since $R$ is noetherian, there are only finitely many such minimal primes. By Krull's principal ideal theorem, $\fm$ is contained in the union of all height $1$ prime ideals. Since there are only finitely many such, prime avoidance implies that $\fm$ is contained in a height $1$ prime ideal, a contradiction since $\dim R > 1$.
\end{proof}

\begin{prop}\label{intersection of height d-1 primes}
Let $(R, \fm)$ be a catenary local domain of dimension $d > 0$. Let $t$ be an integer such that $0\leq t\leq d-1$. Then
\[
\bigcap_{\substack{\fp \subseteq R \, \mathrm{prime}\\{\hgt \fp = t}}} \fp = 0.
\]
\end{prop}

\begin{proof}
If $d=1$, then the only possible value for $t$ is 0 and our conclusion is clear. We will proceed by induction on $d$. When $d=2$,  our conclusion is clear from Lemma \ref{intersection of height 1 primes}. Now suppose that $d \geq 3$ and fix a height $1$ prime ideal $\fp \subseteq R$. By Lemma \ref{intersection of height 1 primes}, we may assume that $t\geq 2$. Since $R$ is catenary, the height $t-1$ prime ideals of $R/\fp$ are precisely the height $t$ prime ideals of $R$ containing $\fp$. Therefore, the inductive hypothesis applied to the $(d-1)$-dimensional ring $R/\fp$ shows that $\fp$ is the intersection of all height $t$ prime ideals of $R$ containing $\fp$. Taking the intersection over all height $1$ prime ideals $\fp$ (which is $0$ by Lemma \ref{intersection of height 1 primes}), we conclude the proof.
\end{proof}

\begin{prop}\label{localization is pseudo-Jacobson}
Let $(R, \fm)$ be a catenary local domain of dimension $d \geq 2$, and let $f \in \fm$ be a nonzero element.
\begin{enumerate}[(a)]
\item The ring $R_f$ is pseudo-Jacobson.
\item Every maximal ideal of $R_f$ has height $d-1$.
\end{enumerate}
\end{prop}

\begin{proof}
We claim first that there are infinitely many height $d-1$ prime ideals of $R$ that do not contain $f$. Suppose otherwise and let $\fp_1, \ldots, \fp_n$ be all the height $d-1$ prime ideals not containing $f$. Choose a nonzero $g \in \cap_{i=1}^n \fp_i$: since $R$ is a domain, it is enough to choose a nonzero element of each $\fp_i$ and let $g$ be their product. Then $fg$ belongs to every height $d-1$ prime ideal of $R$, a contradiction to Proposition \ref{intersection of height d-1 primes}.

Now let $\fp$ be a prime ideal of $R$ such that $\fp R_f$ is a non-maximal prime ideal of $R_f$. Since $\dim R_f = d-1$, the height of $\fp$ is at most $d-2$, so the quotient $R/\fp$ is a catenary local domain of dimension at least $2$. By the reasoning of the previous paragraph applied to $R/\fp$, there are infinitely many height $\dim(R/\fp) - 1$ prime ideals in $R/\fp$ that do not contain the image of $f$ (which is a nonzero element in the maximal ideal $\fm/\fp$), and these prime ideals correspond to infinitely many prime ideals of $R$ that contain $\fp$ but not $f$. Since $R$ is catenary, these prime ideals all have height $d-1$, and therefore correspond to maximal ideals in $R_f$. This proves part (a).

To prove part (b), suppose $\fp$ is a prime ideal of $R$ such that $f \notin \fp$ and $\hgt \fp \leq d-2$. Then as in the proof of part (a), the quotient $R/\fp$ has dimension at least $2$ and satisfies the hypotheses of Proposition \ref{intersection of height d-1 primes}, so $\fp$ is properly contained in infinitely many prime ideals of $R$ that do not contain $f$, and therefore $\fp R_f$ is not a maximal ideal of $R_f$. We conclude that all maximal ideals of $R_f$ must have height $d-1$, as claimed.
\end{proof}

\section{The last terms of minimal injective resolutions}\label{minimal resolutions}

In this section, we study minimal injective resolutions. Proposition \ref{inj hull is finite} below shows that the property of a module $M$ being of finite length as a $D$-module (resp. being $F$-finite) is inherited by the indecomposable summands of the last term of the minimal injective resolution of $M$. The following lemma is the key to proving both cases of this.

\begin{lem}\label{summands of top inj module}
Let $R$ be a noetherian domain and let $M$ be an $R$-module of finite injective dimension $t$. Suppose that for all prime ideals $\fp$ of $R$ and all $i \geq 0$, the local cohomology $R$-modules $H^i_{\fp}(M)$ have finitely many associated primes, and their localizations $(H^i_{\fp}(M))_{\fp}$ are injective $R_{\fp}$-modules. Then for all $\fp \in \Spec(R)$ such that $\mu^t(\fp, M) > 0$, there exists an ideal $J \subseteq R$ such that the quotient 
\[
N = H^t_{\fp}(M)/\Gamma_J(H^t_{\fp}(M))
\]
is isomorphic to a direct sum of copies of $E(R/\fp)$.
\end{lem}

\begin{proof}
Let $\fp \in \Spec(R)$ be such that $\mu^t(\fp, M) > 0$. Since $(H^i_{\fp}(M))_{\fp}$ is injective over $R_{\fp}$ for all $i$, it follows from Lemma \ref{Bass numbers of local coh}(b) that $\mu^0(\fp, H^t_{\fp}(M)) = \mu^t(\fp, M) > 0$, and therefore $\fp \in \Ass H^t_{\fp}(M)$. By hypothesis, $H^t_{\fp}(M)$ has only finitely many associated primes: say $\Ass H^t_{\fp}(M) = \{\fp, \fq_1, \ldots, \fq_r\}$. Let $J = \fq_1 \cdots \fq_r$ (with the convention that if $r = 0$, that is, if $\Ass H^t_{\fp}(M) = \{\fp\}$, then $J = R$), and as in the statement of the lemma, let $N = H^t_{\fp}(M)/\Gamma_J(H^t_{\fp}(M))$. By \cite[Exercise 2.1.14]{LCBookBrodmannSharp}, $\Ass H^t_{\fp}(M)$ is the disjoint union of $\Ass N$ and $\Ass \Gamma_J(H^t_{\fp}(M))$, from which we conclude that $\Ass N = \{\fp\}$. 

Now let $f \in R \setminus \fp$ be given. By hypothesis, the minimal injective resolution $E^{\bullet}(M)$ is a complex of length $t$. Since $f \notin \cup_{\fq \in \Ass N} \fq = \fp$, multiplication by $f$ is injective on $N$. On the other hand, since $R$ is a domain, the injective $R$-module $\Gamma_{\fp}(E^t(M))$ is divisible, so multiplication by any non-zero $f \in R$ is surjective on $\Gamma_{\fp}(E^t(M))$ and therefore on any quotient of $\Gamma_{\fp}(E^t(M))$. Since $H^t_{\fp}(M)$ (and hence $N$) is such a quotient, we see that multiplication by $f$ is an isomorphism on $N$ for all $f \in R \setminus \fp$, and therefore $N = N_{\fp}$. But $(\Gamma_J(H^t_{\fp}(M)))_{\fp} = 0$, so $N_{\fp} = (H^t_{\fp}(M))_{\fp}$, which by hypothesis is an injective $R_{\fp}$-module and is supported only at $\fp R_{\fp}$. We conclude that $N = N_{\fp}$ is isomorphic to a direct sum of copies of $E_{R_{\fp}}(R_{\fp}/\fp R_{\fp})$; but $E_{R_{\fp}}(R_{\fp}/\fp R_{\fp}) \cong E(R/\fp)$ as $R$-modules, so the lemma follows.
\end{proof} 

\begin{prop}\label{inj hull is finite}
\begin{enumerate}[(a)]
\item Let $R = k[[x_1, \ldots, x_n]]$ where $k$ is a field of characteristic zero, let $M$ be a holonomic $D(R,k)$-module, and let $t = \injdim_R M$. For any $\fp \in \Spec R$ such that $\mu^t(\fp, M) > 0$, the $D(R,k)$-module $E(R/\fp)$ is holonomic.
\item Let $R = k[[x_1, \ldots, x_n]]$ where $k$ is a field of characteristic zero, let $M$ be a holonomic $D(R,k)$-module, let $S \subseteq R$ be a multiplicative subset, and let $t = \injdim_{S^{-1}R} S^{-1}M$. For any $S^{-1}\fp \in \Spec S^{-1}R$ such that $\mu^t(S^{-1}\fp, S^{-1}M) > 0$, the $D(S^{-1}R, k)$-module $E_{S^{-1}R}(S^{-1}R/S^{-1}\fp)$ is of finite length.
\item Let $R$ be a noetherian regular domain of characteristic $p>0$, let $M$ be an $F$-finite $F$-module, and let $t = \injdim_R M$. For any $\fp \in \Spec(R)$ such that $\mu^t(\fp, M) > 0$, the $F$-module $E(R/\fp)$ is $F$-finite.
\end{enumerate}
\end{prop}

We observe that parts (a) and (b) remain true if $R$ is replaced with a \emph{polynomial} ring (see the proof of \cite[Theorem 4.4]{zhanginjdim}).

\begin{proof}
We prove part (b) first, and we begin by verifying the hypotheses of Lemma \ref{summands of top inj module} for $S^{-1}M$. The ring $S^{-1}R$ is a domain. Since $M$ is a $D$-module, it has finite injective dimension as an $R$-module by Proposition \ref{D-modules omnibus}(a), so $S^{-1}M$ has finite injective dimension as an $S^{-1}R$-module (used implicitly in the statement). For all $i \geq 0$ and for all $S^{-1}\fp \in \Spec S^{-1}R$, $H^i_{\fp}(M)$ is a holonomic $D$-module and so has finitely many associated primes; it follows that $S^{-1}(H^i_{\fp}(M)) \cong H^i_{S^{-1}\fp}(S^{-1}M)$ has finitely many associated primes as an $S^{-1}R$-module. All that remains to be checked is that, for all $i \geq 0$, $(H^i_{S^{-1}\fp}(S^{-1}M))_{S^{-1}\fp}$ is an injective $(S^{-1}R)_{S^{-1}\fp}$-module. Since the ring $(S^{-1}R)_{S^{-1}\fp}$ is simply $R_{\fp}$, and $(H^i_{S^{-1}\fp}(S^{-1}M))_{S^{-1}\fp} \cong H^i_{\fp R_{\fp}}(M_{\fp})$ as $R_{\fp}$-modules, this follows from Lemma \ref{localized local coh is injective}.
 
Now let $S^{-1}\fp \in \Spec S^{-1}R$ be such that $\mu^t(S^{-1}\fp, S^{-1}M) > 0$. By the proof of Lemma \ref{summands of top inj module}, there is an ideal $J \subseteq S^{-1}R$ such that $N = H^t_{S^{-1}\fp}(S^{-1}M)/\Gamma_J(H^t_{S^{-1}\fp}(S^{-1}M))$ is isomorphic to a direct sum of copies of 
\[
E_{R_{\fp}}(R_{\fp}/\fp R_{\fp}) \cong E_{S^{-1}R}(S^{-1}R/S^{-1}\fp)
\]
as $S^{-1}R$-modules and, in fact, as $D(S^{-1}R, k)$-modules. Since $H^t_{\fp}(M)$ is a holonomic (and hence finite length) $D$-module, its localization $H^t_{S^{-1}\fp}(S^{-1}M)$ (and hence the $D(S^{-1}R, k)$-module quotient $N$) is of finite length as a $D(S^{-1}R, k)$-module.  But then $E_{S^{-1}R}(S^{-1}R/S^{-1}\fp)$ must be of finite length as well, completing the proof of part (b). If we do not localize (that is, if $S^{-1}R = R$, $S^{-1}\fp = \fp \subseteq R$, and $S^{-1}M = M$), then the same proof shows that $E(R/\fp)$ is holonomic, proving part (a).

Finally, we prove part (c). Since $M$ is an $F$-finite $F$-module, so are the local cohomology modules $H^i_I(M)$ for all $i \geq 0$ and all ideals $I \subseteq R$; what is more, $(H^i_I(M))_{\fp}$ is an $F_{R_{\fp}}$-finite $F_{R_{\fp}}$-module for all $\fp \in \Spec(R)$, so since $(H^i_{\fp}(M))_{\fp}$ is supported only at the maximal ideal $\fp R_{\fp}$, it is an injective $R_{\fp}$-module by Proposition \ref{F-modules omnibus}(a). Therefore, the hypotheses of Lemma \ref{summands of top inj module} are satisfied, so if $\fp \in \Spec(R)$ is such that $\mu^t(\fp, M) > 0$, then there is an ideal $J \subseteq R$ such that $N = H^t_{\fp}(M)/\Gamma_J(H^t_{\fp}(M))$ is isomorphic to a direct sum of copies of $E(R/\fp)$. The $R$-submodule $\Gamma_J(H^t_{\fp}(M)) \subseteq H^t_{\fp}(M)$ is in fact an $F$-submodule, so the quotient $N$ is an $F$-finite $F$-module. It follows that $E(R/\fp)$ must also be $F$-finite, completing the proof. 
\end{proof}

\begin{remark}\label{holonomic is necessary}
In the proof of Proposition \ref{inj hull is finite}, we used the fact that if $M$ is a holonomic $D$-module, any local cohomology module $H^i_I(M)$ has finitely many associated primes as an $R$-module, for the reason that it is itself a holonomic $D$-module. In fact, any finite length (indeed, finitely generated) $D$-module has finitely many associated primes by \cite[Theorem 2.4(c)]{LyubeznikFinitenessLocalCohomology}. However, we do not know whether $H^i_I(M)$ is of finite length as a $D$-module whenever $M$ is. This is the reason why we require the stronger hypothesis (in Proposition \ref{inj hull is finite} and in our Main Theorem) that $M$ be holonomic.
\end{remark}

Having shown that the minimal injective resolution of an $F$-finite $F$-module or (localization of a) holonomic $D$-module terminates in an object that is the direct sum of indecomposables with certain finiteness properties, our next task is to determine exactly which indecomposables have these finiteness properties. This we do in the following proposition, of which only the $D$-module parts are new.

\begin{prop}\label{which inj hulls are finite}
\begin{enumerate}[(a)]
\item Let $R = k[[x_1, \ldots, x_n]]$ where $k$ is a field of characteristic zero, and let $\fp \subseteq R$ be a prime ideal. The $D(R,k)$-module $E(R/\fp)$ is holonomic if and only if $\hgt \fp \geq n-1$.
\item Let $R = k[[x_1, \ldots, x_n]]$ where $k$ is a field of characteristic zero, let $S \subseteq R$ be a multiplicative subset, and let $S^{-1} \fp \subseteq S^{-1}R$ be a prime ideal. The $D(S^{-1}R, k)$-module $E_{S^{-1}R}(S^{-1}R/S^{-1}\fp)$ is of finite length if and only if $S^{-1}\fp$ is contained in only finitely many distinct prime ideals of $S^{-1}R$.
\item Let $(R, \fm)$ be a regular local ring of characteristic $p>0$ and dimension $n$, and let $\fp \subseteq R$ be a prime ideal.  The $F$-module $E(R/\fp)$ is $F$-finite if and only if $\hgt \fp \geq n-1$.
\item Let $R$ be a noetherian regular ring of characteristic $p>0$, and let $\fp \subseteq R$ be a prime ideal. The $F$-module $E(R/\fp)$ is $F$-finite if and only if $\fp$ is contained in only finitely many distinct prime ideals of $R$.
\end{enumerate}
\end{prop}

In part (a), the conclusion is different from the polynomial case. If $R$ is replaced with a polynomial ring, $E(R/\fp)$ is holonomic if and only if $\fp$ is \emph{maximal}: see \cite[Propositions 4.2, 4.3]{zhanginjdim}. 

\begin{proof}
We prove part (b) first, and we begin by considering the possible cases. Let $h$ denote the height of $S^{-1}\fp$. If $S^{-1}\fp$ is a maximal ideal of $S^{-1}R$, we must show that $E_{S^{-1}R}(S^{-1}R/S^{-1}\fp)$ is of finite length. If there exists a chain $S^{-1}\fp \subset S^{-1}\fs \subset S^{-1}\fq$ of proper inclusions of prime ideals, then since $S^{-1}R$ is a noetherian ring, it is well-known that there are infinitely many prime ideals lying strictly between $S^{-1}\fp$ and $S^{-1}\fq$. Therefore, in this case we must show that $E_{S^{-1}R}(S^{-1}R/S^{-1}\fp)$ is not of finite length (this is the last case we treat below). Since $S^{-1}R$ is regular and therefore catenary, the only remaining case is that in which $S^{-1}\fp$ is not maximal, but the only prime ideals properly containing it are maximal ideals of height $h+1$. In this case, we must show that $E_{S^{-1}R}(S^{-1}R/S^{-1}\fp)$ is of finite length if and only if there are only finitely many such maximal ideals.

Suppose first that $S^{-1}\fp$ is a maximal ideal of $S^{-1}R$. Since $S^{-1}R$ is Gorenstein, by Proposition \ref{local coh omnibus}(d), $E_{S^{-1}R}(S^{-1}R/S^{-1}\fp) \cong H^h_{S^{-1}\fp}(S^{-1}R)$, which is a localization of the holonomic $D$-module $H^h_{\fp}(R)$ and is therefore of finite length as a $D(S^{-1}R, k)$-module by Proposition \ref{D-modules omnibus}(b,d).

Now suppose that $S^{-1}\fp$ is not maximal, but that all maximal ideals containing it have height $h+1$. Since $E^{\bullet}$ can be identified with the Cousin complex of $R$ \cite[Theorem 5.4]{SharpCousinComplex}, all of whose differentials are direct sums of canonical localization maps, it is a complex of $D$-modules. If we localize $E^{\bullet}$ at $S$, we obtain the minimal injective resolution of $S^{-1}R$ as a module over itself, and this resolution is a complex of $D(S^{-1}R, k)$-modules. After applying $\Gamma_{S^{-1}\fp}$, we obtain a short exact sequence
\[
0 \rightarrow H^h_{S^{-1}\fp}(S^{-1}R) \rightarrow E_{S^{-1}R}(S^{-1}R/S^{-1}\fp) \rightarrow \bigoplus_{\substack{S^{-1}\fp \subseteq S^{-1}\fq \\ \hgt S^{-1}\fq = h+1}} E_{S^{-1}R}(S^{-1}R/S^{-1}\fq) \rightarrow 0,
\]
which is an exact sequence of $D(S^{-1}R, k)$-modules. (The last map is surjective by the Hartshorne-Lichtenbaum vanishing theorem \cite[Theorem 8.2.1]{LCBookBrodmannSharp}.) Since $H^h_{\fp}(R)$ is a holonomic $D$-module, it has finite length as a $D$-module, and therefore its localization $H^h_{S^{-1}\fp}(S^{-1}R)$ has finite length as a $D(S^{-1}R, k)$-module. It follows that $E_{S^{-1}R}(S^{-1}R/S^{-1}\fp)$ is of finite length as a $D(S^{-1}R, k)$-module if and only if the third term in the displayed short exact sequence is. By the previous paragraph, each summand $E_{S^{-1}R}(S^{-1}R/S^{-1}\fq)$ is of finite length, since the $S^{-1}\fq$ are maximal ideals; therefore the sum is of finite length if and only if there are finitely many summands, as desired.

Finally, we suppose that there exists a chain $S^{-1}\fp \subset S^{-1}\fs \subset S^{-1}\fq$ of proper inclusions of prime ideals in $S^{-1}R$. We claim that if such a chain exists, $E_{S^{-1}R}(S^{-1}R/S^{-1}\fp)$ cannot be of finite length. We may assume that the chain is saturated, from which it follows that $\hgt S^{-1}\fq = h+2$. If we localize at $S^{-1}\fq$, the ring $(S^{-1}R)_{S^{-1}\fq}$ is isomorphic to $R_{\fq}$, and $E_{S^{-1}R}(S^{-1}R/S^{-1}\fp) \cong E_{R_{\fq}}(R_{\fq}/\fp R_{\fq})$ as $R_{\fq}$-modules. If $E_{S^{-1}R}(S^{-1}R/S^{-1}\fp)$ were of finite length as a $D(S^{-1}R, k)$-module, its localization $E_{R_{\fq}}(R_{\fq}/\fp R_{\fq})$ would be of finite length as a $D(R_{\fq}, k)$-module, so it suffices to prove this last statement false. We have therefore reduced the proof to the case where $S^{-1}R = R_{\fq}$ for some prime ideal $\fq$ and $\fp R_{\fq}$ is a prime ideal in $R_{\fq}$ of height $h = \dim R_{\fq} - 2$.

Let $E^{\bullet} = E^{\bullet}_{R_{\fq}}(R_{\fq})$ be the minimal injective resolution of $R_{\fq}$. The complex $\Gamma_{\fp R_{\fq}}(E^{\bullet})$ takes the form
\[
0 \rightarrow E(R_{\fq}/\fp R_{\fq}) \xrightarrow{\delta^h} \bigoplus_{\substack{\fp R_{\fq} \subseteq \fs R_{\fq} \\ \hgt \fs R_{\fq} = h+1}} E(R_{\fq}/\fs R_{\fq}) \xrightarrow{\delta^{h+1}} E(R_{\fq}/\fq R_{\fq}) \rightarrow 0,
\]
and gives rise to three short exact sequences

\begin{align*}
0 \rightarrow H^h_{\fp R_{\fq}}(R_{\fq}) \rightarrow E_{R_{\fq}}(R_{\fq}/\fp R_{\fq}) \rightarrow \im \delta^h \rightarrow 0, \\
0 \rightarrow \im \delta^h \rightarrow \ker \delta^{h+1} \rightarrow H^{h+1}_{\fp R_{\fq}}(R_{\fq}) \rightarrow 0,\\
0 \rightarrow \ker \delta^{h+1} \rightarrow \bigoplus_{\substack{\fp R_{\fq} \subseteq \fs R_{\fq} \\ \hgt \fs R_{\fq} = h+1}} E_{R_{\fq}}(R_{\fq}/\fs R_{\fq}) \rightarrow E_{R_{\fq}}(R_{\fq}/\fq R_{\fq}) \rightarrow 0, 
\end{align*}
where now the $\delta^j$ are the differentials in the complex $\Gamma_{\fp R_{\fq}}(E^{\bullet})$ (and the third sequence is exact by the Hartshorne-Lichtenbaum theorem). What is more, these are exact sequences of $D(R_{\fq}, k)$-modules, since they arise from localizations of the Cousin complex of $R$. The modules $H^h_{\fp R_{\fq}}(R_{\fq})$, $H^{h+1}_{\fp R_{\fq}}(R_{\fq})$, and $E_{R_{\fq}}(R_{\fq}/\fq R_{\fq})$ are localizations at $\fq$ of holonomic (hence finite length) $D$-modules ($H^h_{\fp}(R)$, $H^{h+1}_{\fp}(R)$, and $H^{h+2}_{\fq}(R)$ respectively), so all three are of finite length as $D(R_{\fq}, k)$-modules. Assume for the purposes of contradiction that $E_{R_{\fq}}(R_{\fq}/\fp R_{\fq})$ is of finite length as a $D(R_{\fq}, k)$-module. Then we have the following chain of implications: since $H^h_{\fp R_{\fq}}(R_{\fq})$ and $E_{R_{\fq}}(R_{\fq}/\fp R_{\fq})$ are of finite length, so is $\im \delta^h$; since $\im \delta^h$ and $H^{h+1}_{\fp R_{\fq}}(R_{\fq})$ are of finite length, so is $\ker \delta^{h+1}$; since $\ker \delta^{h+1}$ and $E_{R_{\fq}}(R_{\fq}/\fq R_{\fq})$ are of finite length, so is $\oplus_{\fp R_{\fq} \subseteq \fs R_{\fq}, \hgt \fs R_{\fq} = h+1} E_{R_{\fq}}(R_{\fq}/\fs R_{\fq})$. This last statement is absurd, since there are infinitely many distinct summands $E_{R_{\fq}}(R_{\fq}/\fs R_{\fq})$. This contradiction completes the proof of part (b).

If we do not localize (that is, if $S^{-1}R = R$ and $S^{-1}\fp = \fp \subseteq R$), then the same proof (using Proposition \ref{D-modules omnibus}(e)) shows that $E(R/\fp)$ is holonomic if and only if $\fp$ is contained in only finitely many distinct prime ideals of $R$. Since $R$ is a local ring of dimension $n$, this condition is satisfied if and only if the height of $R$ is at least $n-1$, proving part (a).

The possible cases in part (d) are the same as in part (b): we must show that $E(R/\fp)$ is $F$-finite whenever $\fp$ is a maximal ideal (which, since $R$ is Gorenstein, follows at once from Proposition \ref{local coh omnibus}(c,d)); that $E(R/\fp)$ is not $F$-finite whenever there exists a chain $\fp \subset \fs \subset \fq$ of proper inclusions of prime ideals (which is \cite[Proposition 3.2]{zhanginjdim}); and that in the only remaining case, where $\fp$ is not maximal but the only prime ideals properly containing it are maximal ideals of height $\hgt \fp + 1$, that $E(R/\fp)$ is $F$-finite if and only if there are only finitely many such maximal ideals. This last case is \cite[Proposition 3.1]{zhanginjdim}, which finishes the proof of part (d). As part (c) is merely a special case of part (d), the proof is complete.
\end{proof}

\begin{remark}\label{non-holonomic localization}
In the setting of Proposition \ref{which inj hulls are finite}(a), if $M$ is a holonomic $D$-module and $f \in R$, then $M_f$ is a holonomic $D$-module \cite[Theorem 3.4.1]{BjorkBookRIngDiffOperators}. It is known that $M_{\fp}$ need not be a holonomic $D$-module for all prime ideals $\fp \subseteq R$. We remark that the proposition provides many such examples: let $\fp$ be any prime ideal of height $h \leq n-2$, and consider the holonomic $D$-module $H^h_{\fp}(R)$. Its localization at $\fp$ is isomorphic to $E(R/\fp)$, which is not a holonomic $D$-module by the proposition.
\end{remark}

We record separately the special cases of Proposition \ref{which inj hulls are finite} that we will use in the proof of our Main Theorem. This is where the pseudo-Jacobson property defined in section \ref{Pseudo-Jacobson rings} is used.

\begin{cor}\label{special cases}
\begin{enumerate}[(a)]
\item Let $R = k[[x_1, \ldots, x_n]]$ where $k$ is a field of characteristic zero and $n \geq 2$. If $\fq \subseteq R$ is a prime ideal of height $h \geq 2$, $f \in \fq R_{\fq}$ is a nonzero element, and $(\fp R_{\fq})_f$ is a prime ideal of $(R_{\fq})_f$, then the $D((R_{\fq})_f, k)$-module $E_{(R_{\fq})_f}((R_{\fq})_f/(\fp R_{\fq})_f)$ has finite length if and only if $(\fp R_{\fq})_f$ is a maximal ideal in $(R_{\fq})_f$, that is, if and only if $\fp$ is a height $h-1$ prime ideal of $R$ contained in $\fq$ and not containing $f$.
\item Let $(R, \fm)$ be a regular local ring of characteristic $p>0$ and dimension $n \geq 2$. If $f \in \fm$ is a nonzero element and $\fp R_f \subseteq R_f$ is a prime ideal, the $F_{R_f}$-module $E_{R_f}(R_f/\fp R_f)$ is $F_{R_f}$-finite if and only if $\fp R_f$ is a maximal ideal in $R_f$, that is, if and only if $\fp$ is a height $n-1$ prime ideal of $R$ not containing $f$.
\end{enumerate}
\end{cor}

\begin{proof}
A regular local ring is a catenary domain, so by Proposition \ref{localization is pseudo-Jacobson}, the rings $R_f$ of part (b) and $(R_{\fq})_f$ of part (a) are pseudo-Jacobson, and all their maximal ideals have the same height $n-1$. By the pseudo-Jacobson property, every non-maximal prime ideal $\fp R_f$ of $R_f$ in part (b) (resp. every non-maximal prime ideal $(\fp R_{\fq})_f$ of $(R_{\fq})_f$ in part (a)) is contained in infinitely many distinct maximal ideals, so part (b) (resp. (a)) follows from Proposition \ref{which inj hulls are finite}(d) (resp. (b)).
\end{proof}

\section{Injective dimension of holonomic $D$-modules}\label{injdim over power series}

In this section, we prove the characteristic-zero part (Theorem \ref{holonomic over power series}) of our main theorem. Most of the work in the proof is contained in the following proposition.

\begin{prop}\label{injdim equals dim after localizing D-mod}
Let $R = k[[x_1, \ldots, x_n]]$ where $k$ is a field of characteristic zero, and let $M$ be a holonomic $D(R,k)$-module. Let $\fq$ be a prime ideal of $R$ belonging to $\Supp_R M$, and let $f \in \fq R_{\fq}$ be a element that does not belong to any minimal prime of $M_{\fq}$. Then
\[
\injdim_{(R_{\fq})_f} (M_{\fq})_f = \dim \Supp_{(R_{\fq})_f} (M_{\fq})_f.
\]
\end{prop}

\begin{proof}
Recall that if $S \subseteq R$ is a multiplicative subset, then
\[
\injdim_{S^{-1}R} S^{-1}M \leq \dim \Supp_{S^{-1}R} S^{-1}M
\] 
by Theorem \ref{inj dim bound for arbitrary localization}. We will use this fact repeatedly below.

We proceed by induction on $\dim \Supp_{R_{\fq}} M_{\fq}$. Observe first that if $\dim R_{\fq} < 2$, then either $\dim \Supp_{(R_{\fq})_f} (M_{\fq})_f = 0$ and the statement is immediate by the previous paragraph, or no such $f$ as in the statement exists. Therefore we may assume that $\hgt \fq \geq 2$ for all prime ideals $\fq$ we encounter. Let $\fq$ be a minimal element of $\Supp_R M$. The localization $M_{\fq}$ has zero-dimensional support over $R_{\fq}$, so it is an injective $R_{\fq}$-module; the further localization $(M_{\fq})_f$ for any $f \in \fq R_{\fq}$ is then an injective $(R_{\fq})_f$-module, establishing the base case.

Now suppose that $l \geq 0$ and that the displayed equality holds for all $\fp \in \Supp_R M$ such that $\dim \Supp_{R_{\fp}} M_{\fp} \leq l$. Fix $\fq \in \Supp_R M$ such that $\dim \Supp_{R_{\fq}} M_{\fq} = l + 1$, and let $f \in \fq R_{\fq}$ be an element that does not belong to any minimal prime of $M_{\fq}$. Choose $\fp R_{\fq} \in \Supp_{R_{\fq}} M_{\fq}$ such that $\dim \Supp_{R_{\fp}} M_{\fp} = \dim \Supp_{R_{\fq}} M_{\fq} - 1 = l$ and $f \in \fp R_{\fq}$. Then $f$ does not belong to any minimal prime of $M_{\fp}$, so $\dim \Supp_{(R_{\fp})_f} (M_{\fp})_f = \dim \Supp_{R_{\fp}} M_{\fp} - 1 = l - 1$ ($R_{\fp}$ is a local ring) and by the inductive hypothesis,
\[
\injdim_{(R_{\fp})_f} (M_{\fp})_f = \dim \Supp_{(R_{\fp})_f} (M_{\fp})_f = l-1.
\]
Since $(R_{\fp})_f$ is a localization of $(R_{\fq})_f$, we obtain the chain of inequalities
\[
l-1 = \injdim_{(R_{\fp})_f} (M_{\fp})_f \leq \injdim_{(R_{\fq})_f} (M_{\fq})_f \leq \dim \Supp_{(R_{\fq})_f} (M_{\fq})_f = l.
\]
It remains only to rule out the case $\injdim_{(R_{\fq})_f} (M_{\fq})_f = l-1$. Since $\injdim_{(R_{\fp})_f} (M_{\fp})_f = l-1$, there is a prime ideal $(\fs R_{\fp})_f$ of $(R_{\fp})_f$ such that 
\[
\mu^{l-1}((\fs R_{\fp})_f, (M_{\fp})_f) \, (= \mu^{l-1}((\fs R_{\fq})_f, (M_{\fq})_f)) > 0.
\]
Since $f \in \fp R_{\fq} \setminus \fs R_{\fq}$ and $\fp R_{\fq} \subset \fq R_{\fq}$, we have $\hgt \fs \leq \hgt \fq - 2$. It follows that $(\fs R_{\fq})_f$ is not a maximal ideal of $(R_{\fq})_f$. Since $\mu^{l-1}((\fs R_{\fq})_f, (M_{\fq})_f)) > 0$, we may invoke Proposition \ref{inj hull is finite}(b) and Corollary \ref{special cases}(a), which here imply that $\injdim_{(R_{\fq})_f} (M_{\fq})_f$ cannot equal $l-1$, completing the proof.
\end{proof}

\begin{thm}\label{holonomic over power series}
Let $R = k[[x_1, \ldots, x_n]]$ where $k$ is a field of characteristic zero, and let $M$ be a holonomic $D(R,k)$-module. Then
\[
\injdim_R M \geq \dim \Supp_R M - 1.
\]
\end{thm}

\begin{proof}
We may assume that $\dim \Supp_R M \geq 2$, as otherwise there is nothing to prove. Since $M$ is holonomic, it has finitely many associated primes as an $R$-module. By prime avoidance, we can choose a nonzero element $f \in \fm$ that does not belong to any minimal prime of $M$. We have $\dim \Supp_{R_f} M_f = \dim \Supp_R M - 1$. By Proposition \ref{injdim equals dim after localizing D-mod} (applied to $\fq = \fm$), we have $\injdim_{R_f} M_f = \dim \Supp_{R_f} M_f$. Since $\injdim_R M \geq \injdim_{R_f} M_f$, the theorem follows.
\end{proof}

\begin{remark}\label{bound is sharp}
The lower bound in Theorem \ref{holonomic over power series} is the best possible. Indeed, let $\fp \subseteq R$ be a prime ideal of height $n-1$, and let $E(R/\fp)$ be the injective hull of $R/\fp$. By Proposition \ref{which inj hulls are finite}(a), $E(R/\fp)$ is a holonomic $D$-module, yet we have $\injdim_R E(R/\fp) = 0$ and $\dim \Supp_R E(R/\fp) = 1$. As shown by Hellus in \cite[Example 2.9]{hellus}, this example can be realized as a local cohomology module of $R$. Take $n=3$ and let $I = (x_1x_2, x_1x_3)$ and $\fp = (x_2, x_3)$. Then $\hgt \fp = n-1 = 2$ and the holonomic $D$-module $H^2_I(R)$ is isomorphic to $E(R/\fp)$, therefore has injective dimension equal to one less than the dimension of its support. 
\end{remark}

\section{Injective dimension of $F$-finite $F$-modules}\label{injdim of F-modules}

In this section, we prove the positive-characteristic part (Theorem \ref{F-finite over regular char p}) of our main theorem. We begin with a counterpart to Proposition \ref{injdim equals dim after localizing D-mod}.  

\begin{prop}\label{injdim equals dim after localizing F-mod}
Let $(R, \fm)$ be a regular local ring of characteristic $p > 0$, and let $M$ be an $F$-finite $F$-module. Let $\fq$ be a prime ideal of $R$ belonging to $\Supp_R M$, and let $f \in \fq R_{\fq}$ be a element that does not belong to any minimal prime of $M_{\fq}$. Then
\[
\injdim_{(R_{\fq})_f} (M_{\fq})_f = \dim \Supp_{(R_{\fq})_f} (M_{\fq})_f.
\]
\end{prop}

\begin{proof}
If $S \subseteq R$ is a multiplicative subset, then
\[
\injdim_{S^{-1}R} S^{-1}M \leq \dim \Supp_{S^{-1}R} S^{-1}M
\] 
by Proposition \ref{F-modules omnibus}(a,c). The proof is now word-for-word the same as the proof of Proposition \ref{injdim equals dim after localizing D-mod}, except that we use part (b) of Corollary \ref{special cases} instead of part (a).
\end{proof}

\begin{thm}\label{F-finite over regular char p}
Let $R$ be a noetherian regular ring of characteristic $p > 0$, and let $M$ be an $F$-finite $F$-module. Then
\[
\injdim_R M \geq \dim \Supp_R M - 1.
\]
\end{thm}

For the same reasons as in Remark \ref{bound is sharp}, the lower bound in Theorem \ref{F-finite over regular char p} is the best possible.

\begin{proof}
We may assume that $\dim \Supp_R M \geq 2$, as otherwise there is nothing to prove. We may also assume that $(R, \fm)$ is local; if the local case is known, we may choose a maximal ideal $\fm$ in $\Supp_R M$ such that $\dim \Supp_R M = \dim \Supp_{R_{\fm}} M_{\fm}$, and we have
\[
\dim \Supp_R M - 1 = \dim \Supp_{R_{\fm}} M_{\fm} - 1 \leq \injdim_{R_{\fm}} M_{\fm} \leq \injdim_R M,
\]
so that the global case follows. Since $M$ is $F$-finite, it has finitely many associated primes as an $R$-module. By prime avoidance, we can choose a nonzero element $f \in \fm$ that does not belong to any minimal prime of $M$. We have $\dim \Supp_{R_f} M_f = \dim \Supp_R M - 1$. By Proposition \ref{injdim equals dim after localizing F-mod} (applied to $\fq = \fm$), we have $\injdim_{R_f} M_f = \dim \Supp_{R_f} M_f$. Since $\injdim_R M \geq \injdim_{R_f} M_f$, the theorem follows.
\end{proof}

Let $(R,\fm)$ be a regular local ring of characteristic $p > 0$, and let $M$ be an $F$-finite $F$-module. Set $n=\dim R$, $d=\dim \Supp_R M$, and $t=\injdim_R M$. We know by Theorem \ref{F-finite over regular char p} that $t \in \{d-1, d\}$. We also know by Propositions \ref{inj hull is finite}(c) and \ref{which inj hulls are finite}(c) that if $\fp \subseteq R$ is a prime ideal such that $\mu^t(\fp, M) > 0$, then $\hgt \fp \in \{n-1, n\}$. 

It is easy to see that if $t = d$, then $\mu^t(\fp, M) > 0$ if and only if $\fp = \fm$: indeed, if $\mu^t(\fp, M) > 0$ for some non-maximal prime ideal $\fp$ in the support of $M$, we can localize at $\fp$, obtaining an $F_{R_{\fp}}$-module $M_{\fp}$ whose injective dimension is still $d$ but whose support has dimension strictly less than $d$, in contradiction to Proposition \ref{F-modules omnibus}(a).

\begin{question}\label{last inj res term only supported at maximal ideal}
Does the converse hold? That is, if $\mu^t(\fp, M) > 0$ only for $\fp = \fm$, must we have $t = d$?
\end{question}

One can also ask the analogous question for holonomic $D$-modules over formal power series rings.

A positive answer to Question \ref{last inj res term only supported at maximal ideal} would impose strong constraints on the form of the minimal injective resolution $E^{\bullet}(M)$. In particular, it follows from an easy induction argument that we would have $\dim \Supp_R E^i(M) = \dim \Supp_R M - i$ for $0 \leq i \leq t$. 

\bibliographystyle{plain}

\bibliography{injectivedimension2}

\end{document}